\documentclass[10pt,a4paper]{article}
\usepackage[latin1]{inputenc}
\usepackage{amsmath}
\usepackage{amsfonts}
\usepackage{amssymb}
\usepackage{amsthm}
\usepackage{mathrsfs}
\usepackage{latexsym}
\usepackage{amsmath,amscd}

\newtheorem{definition}{Definition}

\newtheorem{theorem}[definition]{Theorem}
\newtheorem{lemma}[definition]{Lemma}
\newtheorem{proposition}[definition]{Proposition}

\newtheorem{remark}[definition]{Remark}
\newtheorem*{main}{Theorem A}
\newtheorem*{main1}{Corollary C}
\newtheorem*{main2}{Theorem B}
\newtheorem*{main3}{Corollary D}
\newtheorem*{main4}{Theorem E}
\newtheorem*{main5}{Theorem F}
\newtheorem*{main6}{Example G}

\newtheorem*{main8}{Lemma \cite[Lemma 2.3]{NowakSpakula}}
\newtheorem*{main9}{Theorem \cite[Lemma 4.1, Theorem 4.2]{NowakSpakula}}

\def\Xint#1{\mathchoice
{\XXint\displaystyle\textstyle{#1}}%
{\XXint\textstyle\scriptstyle{#1}}%
{\XXint\scriptstyle\scriptscriptstyle{#1}}%
{\XXint\scriptscriptstyle\scriptscriptstyle{#1}}%
\!\int}
\def\XXint#1#2#3{{\setbox0=\hbox{$#1{#2#3}{\int}$ }
\vcenter{\hbox{$#2#3$ }}\kern-.6\wd0}}

\def\dashint{\Xint-}

\begin{document}

\section*{Characterising Sobolev inequalities by controlled coarse homology and applications for hyperbolic spaces}

\begin{Large}
{Juhani Koivisto} \\ 
\end{Large} \\ \emph{Department of Mathematics and Statistics, University of Helsinki \\ address: P.O. Box 68 FI-00014 University of Helsinki \\ email: juhani.koivisto@helsinki.fi} \\

\textbf{Abstract.} 
We give a Sobolev inequality characterisation for the vanishing of a fundamental class in the controlled coarse homology of Nowak and \v{S}pakula for quasiconvex uniform spaces that support a local weak $(1,1)$-Poincar\'e inequality. As applications, we consider visual Gromov hyperbolic spaces and Carnot groups. \\  

\textbf{Mathematics subject classification (2000):}  Primary 53C23, 30L99; Secondary 58J32 \\

\textbf{Key words:} Controlled coarse homology, Sobolev inequalities. \\

\section{Introduction} \label{introduction}
In this article, a \emph{metric measure space} $(X,d,\mu)$ is a metric space $(X,d)$ with Borel regular outer measure $\mu$ such that $\mu(X) > 0$ and $\mu(B(x,r))< \infty$ for every $x \in X$ and $r > 0$. In what follows, we call a function $\varrho \colon [0,\infty) \rightarrow [0,\infty)$ a \emph{control function} if it is non-decreasing, $\varrho(0)=1$, and satisfies the conditions
\begin{align}
\varrho(\varepsilon +t) &\leq L(\varepsilon) \varrho(t) \label{L} \tag{$\varrho_1$}
\end{align} and \begin{align} \varrho(\varepsilon t) &\leq M(\varepsilon)\varrho(t), \label{M} \tag{$\varrho_2$}
\end{align}
for some functions $L,M \colon (0,\infty) \rightarrow (0,\infty)$ whenever $t, \varepsilon > 0$. The space $(X,d,\mu)$ satisfies the \emph{global} $\varrho$-\emph{weighted $(1,1)$-Sobolev inequality $(S_{1,1}^\varrho)$} if for some control function $\varrho$ there exists $C > 0$ and $o \in X$ such that 
\begin{align}
\int_X \vert u\vert d \mu \leq C \int_X \vert \nabla u \vert \varrho(d(o,\cdot)) d \mu \nonumber \end{align} for every $u \in N^{1,1}(X,d,\mu)$ with bounded support. Here, $N^{1,1}(X,d,\mu)$ is the \emph{Newton-Sobolev space} of equivalence classes of integrable functions $u \colon X \rightarrow [-\infty, \infty]$ with integrable upper gradient, and $\vert \nabla u\vert \colon X \rightarrow [0,\infty]$ the $1$-weak minimal upper gradient of $u$; see \cite[Section 7.1]{Heinonen}. If $(S_{1,1}^\varrho)$ holds for $\varrho \equiv 1$ we say that $(X,d,\mu)$ satisfies $(S_{1,1})$. \par  
Given $(X,d,\mu)$ what is the relationship between the $\varrho$-isoperimetry of $(X,d)$ and $(S_{1,1}^\varrho)$? Previously, $\varrho$-isoperimetry has been studied in \cite{Erschler, NowakSpakula, Zuk}, and Sobolev inequalities including $(S_{1,1})$ for Riemannian manifolds in \cite{Coul4, Coul2, Coul} and for metric measure spaces in \cite{koskela, RT}. Our main result is a metric measure version of a result of Nowak and \v{S}pakula \cite[Theorem 4.2]{NowakSpakula}, in part saying the following.   
\begin{main} \label{theoremA}
Let $(X,d,\mu)$ be a quasiconvex uniform space supporting a local weak $(1,1)$-Poincar\'e inequality. Then $(X,d,\mu)$ satisfies $(S_{1,1}^\varrho)$ for a given control function $\varrho$ if and only if $0=[\Gamma] \in H^\varrho_0(\Gamma)$ for any quasi-lattice $\Gamma \subseteq X$.\end{main}

For the proof of Theorem A as well as several other equivalent statements; see Theorem \ref{ADLAD}. For the terminology related to controlled coarse homology $H^\varrho_*$ and its relationship with $\varrho$-isoperimetry, we refer to Section \ref{hom}. A \emph{quasi-lattice} in $(X,d)$ is any \emph{(C-)cobounded} set $\Gamma \subseteq X$, that is $N_C(\Gamma) := \lbrace x \in X \colon d(x,\Gamma) < C \rbrace = X$ for some $C>0$, which is \emph{uniformly locally finite} in the sense that there exists a function $N \colon (0,\infty) \rightarrow \mathbb{N}$ for which the cardinality $\# (\Gamma \cap B(x,r)) \leq N(r)$ for every open ball $B(x,r) \subseteq X$. The space $(X,d)$ is \emph{(Q-)Quasiconvex} if there exists $Q \geq 1$ such that for any $x,y \in X$ there is a rectifiable path $\gamma$ from $x$ to $y$ of length $\ell(\gamma) \leq Qd(x,y)$. A Borel regular outer measure $\mu$ is \emph{uniform} if there exist non-decreasing functions $f,g \colon (0,\infty) \rightarrow (0,\infty)$ such that $f(r) \leq \mu(B(x,r)) \leq g(r)$ for all $0 < r < \infty$ and $B(x,r) \subseteq X$. We say that $(X,d,\mu)$ is uniform if $\mu$ is uniform. Examples of uniform spaces are locally Ahlfors regular spaces and second-countable locally compact compactly generated groups with respect to a left-invariant metric and Haar measure; see \cite[Proposition 4.B.9]{CorHar}. A space $(X,d,\mu)$ supports a \emph{local weak $(1,p)$-Poincar\'e inequality (up to scale $R_P$)} for $1 \leq p < \infty$ if there exist $C_P,R_P> 0$ and $\tau \geq 1$ such that for all $B(x,r) \subseteq X$ with $0<r \leq R_P$, $0 < \mu(B(x,\tau r)) < \infty$ and \begin{align}\dashint_{B(x,r)} \vert u-u_{B(x,r)}\vert d \mu \leq C_P r \left( \dashint_{B(x,\tau r)} g_u^p d \mu \right)^{1/p} \nonumber \end{align} for $u \colon X \rightarrow \mathbb{R}$ such that $u \in {L}^1(B(x, \tau r),d,\mu)$ and its minimal $p$-weak upper gradient of $g_u \colon X \rightarrow [0,\infty]$; this is the local version of \cite[Proposition 8.1.3]{Heinonen}. Here as usual, $$f_A = \dashint_A f d \mu = \dfrac{1}{\mu(A)} \int_A f d \mu $$ assuming that $A \subseteq X$ is a $\mu$-measurable set $0 < \mu (A) < \infty$ and $f \colon A \rightarrow [-\infty, \infty]$ is integrable over $A$. \par 
Contained in the proof of Theorem \ref{ADLAD} is also the following partial result that does not rely on a local weak $(1,1)$-Poincar\'e inequality. 
\begin{main2}
Let $(X,d,\mu)$ be quasiconvex uniform metric measure space satisfying $(S_{1,1}^\varrho)$. Then $[\Gamma]=0$ in $H^\varrho_0(\Gamma)$ for any quasi-lattice $\Gamma \subseteq X$.
\end{main2}

We now list some immediate applications motivating Theorem A.

\begin{main1} \label{yaddayadda}
Let $(X,d,\mu)$ and $(X', d', \mu')$ be quasiconvex uniform spaces that support a local weak $(1,1)$-Poincar\'e inequality. If $(X,d)$ and $(X',d')$ are quasi-isometric, then $(X,d,\mu)$ satisfies $(S_{1,1}^\varrho)$ if and only if $(X',d',\mu')$ satisfies $(S_{1,1}^\varrho)$.  
\end{main1}
Recall that $(X,d)$ and $(X',d')$ are \emph{quasi-isometric} if there exists $f \colon X \rightarrow X'$ and constants $\lambda \geq 1$ and $\mu \geq 0$ such that $$ \lambda^{-1} d(x,x') - \mu \leq d'(f(x),f(x')) \leq \lambda d(x,x') + \mu$$ for all $x,x' \in X$, and $f(X)\subseteq X'$ is $\mu$-cobounded. Corollary C now follows from Theorem A as controlled coarse homology is a quasi-isometry invariant; see \cite[Corollary 2.3]{NowakSpakula}. 
A metric space $(X,d)$ with a quasi-lattice $\Gamma \subseteq X$ is \emph{amenable} if for any $\varepsilon, r > 0$ there exists a non-empty finite $F \subseteq \Gamma$ such that $$\dfrac{\# \partial_r F}{\# F} < \varepsilon,$$ where $\partial_r F = \lbrace x \in \Gamma \colon d(x,\Gamma)<r \textrm{ and } d(x, \Gamma \setminus F) < r\rbrace$. If $(X,d)$ is not amenable, we say that it is \emph{non-amenable}. As observed by Block and Weinberger, a space $(X,d)$ with a quasi-lattice $\Gamma \subseteq X$ is non-amenable if and only if $0=[\Gamma] \in H^1_0(\Gamma)$ where $H^1_0(\Gamma)$ denotes $0$-dimensional controlled coarse homology group for $\varrho \equiv 1$; see \cite[Proposition 2.3, Theorem 3.1]{BlockWeinberger1} as well as \cite{NowakSpakula}. With this in mind, we give the following characterisation. 
\begin{main3} \label{AM}
Let $(X,d,\mu)$ be a quasiconvex uniform space that supports a local weak $(1,1)$-Poincar\'e inequality. Then $(X,d)$ is non-amenable if and only if $(X,d, \mu)$ satisfies $(S_{1,1})$.
\end{main3}
Corollary D follows directly from Theorem A and the characterisation \cite[Theorem 3.1]{BlockWeinberger1}. Note the similarity between Corollary D and \cite[Theorem 7.1]{Coul}; see also \cite[Example 5.8]{RT}.  
\begin{main4} \label{BIGTHEOREM3}
Let $(X,d,\mu)$ be a quasiconvex uniform visual Gromov hyperbolic space defined using the Gromow product. If $(X,d,\mu)$ supports a local weak $(1,1)$-Poincar\'e inequality and its Gromov boundary $\partial X$ is connected and contains at least two points, $(X,d,\mu)$ satisfies $(S_{1,1})$. 
\end{main4}
\begin{proof}
Since $(X,d,\mu)$ is uniform visual and Gromov hyperbolic with connected boundary containing at least two points, $0=[\Gamma] \in H^1_0(\Gamma)$ for any quasi-lattice $\Gamma \subseteq X$; see \cite{Juhani1}. The claim now follows from Theorem A.  
\end{proof}
We give a further application of Corollary E to the Dirichlet problem at infinity that generalises a result of Cao \cite[Corollary 1.1]{Cao}; see also \cite{Ilkka}.  
\begin{main5} \label{DIR}
Suppose $(X,d,\mu)$ is a locally compact quasiconvex visual Gromov hyperbolic metric measure space defined using the Gromov product having uniform measure that supports a local weak $(1,1)$-Poincar\'e inequality. Suppose its Gromov boundary $\partial X$ is connected and contains at least two points. Then, if $f \colon \partial X \rightarrow \mathbb{R}$ is a bounded continuous function, there exists a continuous function $u \colon X^* \rightarrow \mathbb{R}$ on the Gromov closure $X^*$ of $X$ that is $p$-harmonic for $p > 1$ in $X$ and $u \vert \partial X = f$. 
\end{main5}
\begin{proof}
By Theorem E, the space $(X,d,\mu)$ satisfies $(S_{1,1})$ and hence the corresponding $(p,p)$-Sobolev inequality for $1 \leq p < \infty$; see \cite[Example 8]{Ilkka}. By H\"older's inequality, $(X,d,\mu)$ supports a local weak $(1,p)$-inequality for $1 \leq p < \infty$ as well. Thus, $(X,d,\mu)$ satisfies all the assumptions of \cite[Theorem 1.1]{Ilkka} (see Lemma \ref{UNIBG}) and the claim follows.  
\end{proof}
We finish with an example illustrating the case when $\varrho \not \equiv 1$. Write \textbf{$f \preceq g$} for non-decreasing functions $f,g \colon [0,\infty) \rightarrow [0,\infty)$ for which there exist constants $\lambda,\mu >0$ and $c \geq 0$ such that $f(r) \leq \lambda g(\mu r + c)$ for all $r \geq 0$. Also, write $f \lnsim g$ if $f \preceq g$ but $g \not \preceq f$.  
\begin{main6}
The first real Heisenberg group $(\mathcal{H}_1(\mathbb{R}),d_\mathcal{H},\mu)$ with Heisenberg metric satisfies $(S_{1,1}^\varrho)$ for $\varrho(t)=t+1$ but not $(S_{1,1}^{\xi})$ for any other control function $\xi(t) \lnsim t+1$. 
\end{main6}
\begin{proof}
As the first integer Heisenberg group $\mathcal{H}_1(\mathbb{Z}) \leq \mathcal{H}_1(\mathbb{R})$ is a uniform lattice, there exists a quasi-isometry $$f \colon (\mathcal{H}_1(\mathbb{Z}),d_S) \rightarrow (\mathcal{H}_1(\mathbb{R}),d_H)$$ where $d_S$ is the word metric; see \cite[Definition 4.B.1]{CorHar} and \cite[Proposition 5.C.3]{CorHar}. In particular, $H_0^\varrho(\mathcal{H}_1(\mathbb{Z})) \cong H^\varrho_0(\mathcal{H}_1(\mathbb{R}))$ are isomorphic. As the group $\mathcal{H}_1(\mathbb{Z})$ is infinite polycyclic, $0=[\mathcal{H}_1(\mathbb{Z})] \in H^\varrho_0(\mathcal{H}_1(\mathbb{Z}))$ if and only if $\varrho(t)=t+1$; see \cite[Corollary 5.5]{NowakSpakula}. In particular, $0 \neq [\mathcal{H}_1(\mathbb{Z})] \in H^{\xi}_0(\mathcal{H}_1(\mathbb{Z}))$ for $\xi(t) \lnsim t+1$. The claim now follows from Theorem A.\end{proof} 
Similar arguments hold for Carnot groups; again Theorem A gives a homological way to deduce $(S_{1,1}^\varrho)$ from algebraic growth data. 
\\ \\
\textbf{Acknowledgements} I would like to thank Ilkka Holopainen for introducing me to the topic of Sobolev inequalities and for providing me with unpublished notes by Aleksi V\"ah\"akangas on global Sobolev inequalities on Gromov hyperbolic spaces. I also wish to thank Piotr Nowak for many discussions on growth homology, Pekka Pankka whose comments and suggestions more than improved the text at hand, Antti Per\"al\"a for many conversations on related topics, the Technion for its hospitality during my stay from January to May 2014, Uri Bader and Tobias Hartnick for many inspiring conversations, and Eline Zehavi for all her help during this stay. Last, I would like to thank the Academy of Finland, projects 252293, 271983, and the ERC grant 306706, for financial support, and DOMAST for travel support.

\section{Tools of controlled coarse homology} \label{hom}
We first recall some terminology; see \cite{NowakSpakula} for details. A metric space $(X,d)$ is \emph{uniformly coarsely proper} if it has a quasi-lattice $\Gamma \subseteq X$. \par 
\begin{remark} \label{ULFremark}
A metric space $(X,d)$ is uniformly coarsely proper if and only if there exists $r_b > 0$ and $N \colon (r_b, \infty) \times (r_b,\infty) \rightarrow \mathbb{N}$ such that, for all $R > r > r_b$, any open ball of radius $R$  in $X$ can be covered by $N(R,r)$ open balls of radius $r$ in $X$; see \cite[Section 3]{CorHar}.
\end{remark}
A pointed uniformly coarsely proper space $(X,d,o)$ always has a quasi-lattice $\Gamma \ni o$. For $q \in \mathbb{N}$, we denote by $(X^{q+1},d,\bar{o})$ the corresponding pointed $(q+1)$-Cartesian product with basepoint $\bar{o}=(o, \dots, o)$ and metric $$d(\bar{x}, \bar{y}) = \max_{0 \leq i \leq q} d(x_i,y_i)$$ where $\bar{x}=(x_0, \dots, x_q) \in X^{q+1}$ and $\bar{y} = (y_0, \dots, y_q) \in X^{q+1}$.
For a quasi-lattice $\Gamma \ni o$ and a control function $\varrho$, we denote by $C^{\varrho}_q(\Gamma)$ the space of functions $c \colon \Gamma^{q+1} \rightarrow \mathbb{R}$ for which 
\begin{itemize}
\item[(a)] there exists a constant $K(c) \geq 0$, which may depend on $c$, such that $\vert c(\bar{x}) \vert \leq K(c) \varrho(d(\bar{x},\bar o))$ for all $\bar{x} \in \Gamma^{q+1}$;
\item[(b)] $c$ is alternating, that is $c(x_{\sigma(0)}, \dots, x_{\sigma(q)}) = \mathrm{sign}(\sigma)c(x_0, \dots, x_q)$ for all $(x_0, \dots, x_q) \in \Gamma^{q+1}$ and all permutations $\sigma \colon \lbrace 0, \dots, q \rbrace \rightarrow \lbrace 0, \dots, q\rbrace$;
\item[(c)] there exists a constant $P(c) \geq 0$, which may depend on $c$, such that $c(x_0, \dots, x_q) = 0$ if $\max_{i \neq j} d(x_i,x_j) > P(c)$.
\end{itemize}
Note that $C_q^\varrho(\Gamma)$ is an $\mathbb{R}$-module that does not depend on the choice of basepoint by (\ref{L}). A function $c \in C^{\varrho}_q(\Gamma)$ is called a \emph{controlled coarse $q$-chain} and we write $$c= \sum_{(x_0, \dots, x_q) \in \Gamma^{q+1}} c(x_0, \dots, x_q) [x_0, \dots, x_q]$$ where the abstract $q$-cell $[x_0, \dots, x_q] \in C^\varrho_q(\Gamma)$ is the characteristic function $\chi_{(x_0, \dots, x_q)} \colon \Gamma^{q+1} \rightarrow \mathbb{R} $ of the point $(x_0, \dots, x_q)$. 
The \emph{controlled coarse homology} $H^\varrho_*(\Gamma)$ is the homology of the chain complex $$\cdots \xrightarrow{\partial_3} C_2^\varrho(\Gamma) \xrightarrow{\partial_{2}}  C^\varrho_1(\Gamma)\xrightarrow{\partial_1} C^\varrho_0(\Gamma)\xrightarrow{\partial_0} 0$$ where the boundary homomorphism $\partial_q \colon C^\varrho_q(\Gamma) \rightarrow C^\varrho_{q-1}(\Gamma)$ is given by $$\partial_q([x_0, \dots, x_q]) = \sum_{i=0}^q (-1)^i[x_0, \dots,\hat{x}_i, \dots, x_q]$$ for each abstract $q$-cell $[x_0, \dots, x_q]$ and extended linearly to $C^q(\Gamma)$ for $q \in \mathbb{N}\setminus \lbrace 0\rbrace$; as usual, $[x_0, \dots, \hat{x}_i, \dots,x_q]$ denotes the abstract $q$-cell obtained from $[x_0, \dots, x_q]$ by omitting its i:th coordinate. In particular, $\partial_{q-1} \circ \partial_q=0$ and $\partial_q c \in C^\varrho_{q-1}(\Gamma)$ by (\ref{L}). The \emph{$q$-dimensional controlled coarse homology group} is explicitly
$$H^\varrho_q(\Gamma) = \ker \partial_q / \mathrm{im} \, \partial_{q+1}.$$
A special role is played by the homology class $[\Gamma] \in H_0^\varrho(\Gamma)$ of the characteristic function $$\chi_\Gamma = \sum_{x \in \Gamma}  [x] \in C_0^\varrho(\Gamma),$$ called the \emph{fundamental class}. Its vanishing characterises the $\varrho$-isoperimetry of the space. In what follows we use the notation $\vert (x,y) \vert:= d(\bar{o},(x,y))$.
\begin{main9} \label{NSremark}
For a quasi-lattice $\Gamma \ni o$, assume that there exists $C \in (0,1]$ such that $d(x,y) \geq C$ whenever $x,y \in \Gamma$ are distinct, and that for all $x,y \in \Gamma$ there is a sequence $(x=x_0, \dots, x_n=y)$ in $\Gamma$ such that $n \leq d(x,y)$ and $d(x_i,x_{i+1}) \leq 1$ for every $0 \leq i \leq n-1$. Then, the following are equivalent:
\begin{itemize}
\item[(1)] $0=[\Gamma] \in H^{\varrho}_0(\Gamma)$, 
\item[(2)] there exists $C' > 0$ such that for every finitely supported $\eta \colon \Gamma \rightarrow \mathbb{R}$ \begin{align}\sum_{x \in \Gamma} \vert \eta(x)\vert \leq C'\left( \sum_{x \in \Gamma} \sum_ {y \in \overline{B}(x,1)} \vert \eta(x)-\eta(y)\vert \varrho \left(\vert (x,y)\vert\right)\right) \nonumber \end{align} where $\overline{B}(x,1)= \lbrace y \in \Gamma \colon d(x,y) \leq 1\rbrace$, 
\item[(3)] there exists $C''>0$ such that for all finite $F \subseteq \Gamma$ \begin{align}\# F \leq C'' \sum_{x \in \partial F} \varrho(d(o,x)), \nonumber \end{align} where 
$\partial F = \lbrace x \in \Gamma \colon d(x,F)=1 \, \mathrm{or} \, d(x, \Gamma \setminus F) = 1\rbrace.$
\end{itemize}  
\end{main9}

\begin{main8} \label{NS3}
Suppose $\Gamma \subseteq X$ is a quasi-lattice for which there exists $c = \sum_{x \in \Gamma} c(x) [x] \in C^\varrho_0(\Gamma)$ such that $ \inf_{x \in \Gamma} c(x) > 0$ and $[c]=0$ in $H^\varrho_0(\Gamma)$. Then $0 =[\Gamma] \in H^\varrho_0(\Gamma)$. 
\end{main8}
This leads us to the following observation which shows that if $[\Gamma]=0$ for some quasi-lattice $\Gamma \subseteq X$ then $[\Gamma']=0$ for every quasi-lattice $\Gamma' \subseteq X$.
\begin{lemma} \label{vanishingql}
Let $f \colon \Gamma \rightarrow \Gamma'$ be a quasi-isometry between quasi-lattices. Then, $[\Gamma]=0$ in $H^\varrho_0(\Gamma)$ if and only if $[\Gamma']=0$ in $H^\varrho_0(\Gamma')$. 
\end{lemma}
\begin{proof}
The quasi-isometry $f \colon \Gamma \rightarrow \Gamma'$ induces a chain map $f_q \colon C^\varrho_q(\Gamma) \rightarrow C^\varrho_q(\Gamma')$ extending the map $[x_0, \dots,x_q] \mapsto [f(x_0), \dots, f(x_q)]$ linearly to $C^\varrho_q(\Gamma)$. By (\ref{L}) and (\ref{M}), $f_q$ is well-defined. In particular
$$f_0 \left(\sum_{x \in \Gamma} [x] \right) = \sum_{x \in \Gamma} [f(x)] = \sum_{y \in f(\Gamma)} c(y) [y] = c' \in C^\varrho_0(\Gamma')$$ where $c(y) = \# f^{-1}(y) \geq 1$ for $y \in f(\Gamma)$. Since $f(\Gamma) \subseteq \Gamma'$ is a quasi-lattice and $0 = [\Gamma]$ implies that $0 = [c'] \in H^\varrho_0(\Gamma')$ there exists for every $y \in f(\Gamma)$ a controlled coarse $1$-chain $$t_y = \sum_{i=0}^\infty[x_i,x_{i+1}] \in C^\varrho_1(f(\Gamma))$$ where $x_{0} = y$ so that $$t = \sum_{y \in f(\Gamma)} t_y \in C_1^\varrho(f(\Gamma))$$ by the proof of \cite[Lemma 2.3]{NowakSpakula}; see also \cite[Lemma 2.4]{BlockWeinberger1}. By coboundedness, fix $C>0$ such that $N_C(f(\Gamma)) = \Gamma'$. To begin, let $y_1 \in f(\Gamma)$ and let $$t_{w,y_1} = [w,y_1] + t_{y_1} \in C^\varrho_1(\Gamma')$$ for each $w \in B(y_1,C) \setminus \lbrace y_1 \rbrace$. Since $\Gamma'$ is uniformly locally finite, there is at most $\#(B(y_1,C) \cap \Gamma') \leq N(C)$ chains $t_{w,y_1}$. Next, let $y_2 \in f(\Gamma) \setminus \lbrace y_1 \rbrace$ and let $$t_{w,y_2} = [s,y_2] + t_{y_2} \in C^\varrho_1(\Gamma')$$ for each $w \in (B(y_2,C) \setminus \lbrace y_2 \rbrace) \setminus B(y_1,C)$. Again, there is at most $N(C)$ chains $t_{w,y_2}$. Continuing in the obvious way, we obtain a controlled coarse $1$-chain $$t' = \sum_{i=1}^\infty t_{w,y_i} + \sum_{y \in f(\Gamma)} t_y \in C ^\varrho_1(\Gamma')$$ whose boundary is $\partial_1 t' = \sum_{y \in \Gamma'} [y]$. In other words, $0 = [\Gamma'] \in H^\varrho_0(\Gamma')$ as claimed.

\end{proof}

\section{Uniform metric measure spaces, discretisation, and smoothing} 
A metric measure space $(X,d,\mu)$ is a \emph{$(DV)_{\mathrm{loc}}$ space} if it has the \emph{$(DV)_{\mathrm{loc}}$ property} saying that there exists a function $C \colon (0,\infty) \rightarrow (0,\infty)$ such that \begin{align}0 < \mu(B(x,2r)) \leq C(r)\mu(B(x,r)) < \infty \nonumber \end{align} for all $B(x,r) \subseteq X$; see \cite{Coul}. This implies that the space is separable; see \cite[Lemma 3.3.30]{Heinonen}. Examples of $(DV)_{\mathrm{loc}}$ spaces are locally compact groups acting by measure preserving isometries on metric measure spaces \cite[Example 5.4]{RT}, and uniform spaces with $C(r)= g(2r)/f(r)$. 
\subsection{Discretisation and smoothing: from discrete to smooth} \label{onemore}
A \emph{maximal $\varepsilon$-net} in $(X,d)$ is a $\varepsilon$-cobounded subset $N(X,\varepsilon) \subseteq X$ such that $d(x,y) \geq \varepsilon$ whenever $x,y \in N(X,\varepsilon)$ are distinct. We also write $q \sim p$ saying that $q$ is a \emph{neighbour} of $q$ if $p,q \in N(X,\varepsilon)$ and $0 < d(p,q) \leq 3 \varepsilon$.  
By Zorn's lemma, for any $\varepsilon > 0$ and $o \in X \neq \emptyset$ there exists a maximal $\varepsilon$ net $N(X, \varepsilon) \ni o$. \par 
Adapting the argument for doubling spaces in \cite[Section 4.1]{Heinonen}, we record the following fact. 
\begin{remark} \label{heinonenfinal}
A $(DV)_\mathrm{loc}$ space $(X,d,\mu)$ is uniformly coarsely proper as a metric space. In particular any $N(X,\varepsilon)$ is a quasi-lattice.
\end{remark}
 
\begin{lemma} \label{1ST}
Let $(X,d, \mu)$ be an unbounded quasiconvex $(DV)_{\mathrm{loc}}$ space that supports a local weak $(1,1)$-Poincar\'e inequality up to scale $R_P$. Then, given $0< \varepsilon \leq R_P/4$, a quasi-lattice $N(X, \varepsilon) \ni o$, where $\mu(\lbrace o\rbrace)=0$, and a control function $\varrho \colon [0,\infty) \rightarrow [0,\infty)$, there exists  $C>0$ for which $$\sum_{p \in N(X, \varepsilon)} \sum_{q \sim p} \vert u_{B(p,4 \varepsilon)}-u_{B(q,4 \varepsilon)}\vert \varrho(d(o,p)) \mu(B(p, \varepsilon)) \leq C \int_X \vert \nabla u(x)\vert \varrho(d(o,x)) d \mu (x)$$ for every $u \in N^{1,1}(X,d,\mu)$.
\end{lemma}
This lemma is well-known for complete Riemannian manifolds of bounded geometry when $\varrho \equiv 1$ \cite[Lemma 33]{IH}; see also \cite{Kanai2}. Here the point to note is that using inequality (\ref{L}) the classic result can additionally be weighted by the control function $\varrho$ which connects it to controlled coarse homology. \\ \\
\emph{Proof of Lemma \ref{1ST}.}
Let $p \in N(X,\varepsilon)$ and $x \in B(p, 8 \tau \varepsilon)$ where $\tau \geq 1$. Now $d(o,p) \leq d(o,x) + d(x,p) \leq d(o,x) + 8 \tau \varepsilon$, and since $\varrho$ is non-decreasing  
\begin{align} (1) \, \, \, \, \varrho(d(o,p)) \int_{B(p,8 \tau \varepsilon)} \vert \nabla u (x)\vert d \mu (x) &\leq \int_{B(p,8 \tau \varepsilon)} \vert \nabla u (x)\vert \varrho(d(o,x) + 8 \tau \varepsilon) d \mu (x)\nonumber \\  &= \int_{B(p, 8 \tau \varepsilon) \setminus \lbrace o\rbrace} \vert \nabla u(x) \vert \varrho(d(o,x) + 8 \tau \varepsilon) d \mu (x) \nonumber\\ &\leq L(8 \tau \varepsilon) \int_{B(p, 8 \tau \varepsilon)} \vert \nabla u (x)\vert \varrho(d(o,x)) d \mu (x), \nonumber
\end{align}
by (\ref{L}). The proposition follows from estimating (1) from below using the local weak $(1,1)$-Poincar\'e inequality. First, choose a neighbour $q \sim p$ noting that the space is quasiconvex and unbounded. Now $B(p,4 \tau \varepsilon) \cup B(q, 4 \tau \varepsilon) \subseteq B(p, 8 \tau\varepsilon )$ and \begin{align}
\int_{B(p, 8 \varepsilon \tau )} \vert \nabla u (x)\vert d \mu(x) \geq \dfrac{1}{2} \int_{B(p, 4 \tau \varepsilon)} \vert \nabla u(x)\vert d \mu (x) + \dfrac{1}{2}\int_{B(q,4 \tau \varepsilon)} \vert \nabla u (x) \vert d \mu (x) \nonumber. 
\end{align}
By the local weak $(1,1)$-Poincar\'e inequality 
\begin{align}
\dashint_{B(p,4 \tau \varepsilon)} \vert \nabla u(x)\vert d \mu (x) \geq \dfrac{1}{4 \varepsilon C_P}\dashint_{B(p,4 \varepsilon)} \vert u(x) - u_{B(p, 4 \varepsilon)}\vert d \mu(x), \nonumber
\end{align}
and since $\mu(B(p, 4 \tau \varepsilon)) \geq \mu(B(p, 4 \varepsilon))$, \begin{align}
\int_{B(p,4 \tau \varepsilon)} \vert \nabla u(x)\vert d \mu (x) \geq  C \int_{B(p,4 \varepsilon)} \vert u(x) - u_{B(p, 4 \varepsilon)}\vert d \mu(x) \nonumber
\end{align}
for some $C > 0$. Hence
\begin{align}
&\int_{B(p,8 \tau \varepsilon)} \vert \nabla u (x)\vert d \mu (x) \nonumber \\ &\geq \dfrac{1}{2}\int_{B(p, 4 \tau \varepsilon)} \vert \nabla u(x)\vert d \mu (x) + \dfrac{1}{2}\int_{B(q,4 \tau \varepsilon)} \vert \nabla u (x) \vert d \mu (x) \nonumber  \\ &\geq \dfrac{C}{2} \int_{B(p,4 \varepsilon)} \vert u(x) - u_{B(p, 4 \varepsilon)}\vert d \mu(x) + \dfrac{C}{2} \int_{B(q,4 \varepsilon)} \vert u(x) - u_{B(q, 4 \varepsilon)}\vert d \mu(x) \nonumber \\ &\geq \dfrac{C}{2} \int_{B(p,4 \varepsilon) \cap B(q, 4 \varepsilon)} \left( \vert u(x) - u_{B(p, 4 \varepsilon)}\vert + \vert u(x) - u_{B(q, 4 \varepsilon)}\vert \right) d \mu(x) \nonumber \\ &\geq \dfrac{C}{2} \vert u_{B(p, 4 \varepsilon)} - u_{B(q, 4 \varepsilon)}\vert \int_{B(p,\varepsilon)} d \mu(x) \nonumber \\ &= \dfrac{C}{2} \vert u_{B(p, 4 \varepsilon)} - u_{B(q, 4 \varepsilon)}\vert \mu(B(p, \varepsilon)), \nonumber 
\end{align} since $B(p, \varepsilon) \subseteq B(p, 4 \varepsilon) \cap B(q, 4 \varepsilon)$. Using this to estimate (1) gives
\begin{align}
\int_{B(p, 8 \tau \varepsilon)} \vert \nabla u (x)\vert \varrho(d(o,x)) d \mu (x) &\geq \dfrac{\varrho(d(o,p))}{L(8 \tau \varepsilon)} \int_{B(p,8 \tau \varepsilon)}  \vert \nabla u (x)\vert d \mu (x)  \nonumber \\ &\geq \dfrac{C \varrho(d(o,p))}{2L(8 \tau \varepsilon)} \vert u_{B(p, 4 \varepsilon)} - u_{B(q, 4 \varepsilon)}\vert \mu(B(p, \varepsilon)). \nonumber 
\end{align}
Since $N(X,\varepsilon)$ is uniformly locally finite, the number of neighbours $q \sim p$ is uniformly bounded and hence 
\begin{align} 
\int_{B(p, 8 \tau \varepsilon)} \vert \nabla u (x)\vert &\varrho(d(o,x)) d \mu (x) \nonumber \\  &\geq C' \varrho(d(o,p)) \sum_{q \sim p} \vert u_{B(p, 4 \varepsilon)} - u_{B(q, 4 \varepsilon)}\vert \mu(B(p, \varepsilon)). \nonumber 
\end{align}
for some $C'>0$ independent of $u$. Similarly, every $x \in X$ belongs to a uniformly bounded number of open balls of radius $8 \tau \varepsilon$ having a center in $N(X,\varepsilon)$, and altogether
\begin{align}
&\sum_{p \in N(X, \varepsilon)} \sum_{q \sim p} \vert u_{B(p, 4 \varepsilon)} - u_{B(q, 4 \varepsilon)}\vert \varrho(d(o,p))\mu(B(p, \varepsilon)) \nonumber \\ &\leq C'^{-1} \sum_{p \in N(X, \varepsilon)} \int_{B(p, 7 \tau \varepsilon)} \vert \nabla u (x)\vert \varrho(d(o,x))d \mu (x) \nonumber \\ &\leq C'' \int_{X} \vert \nabla u (x)\vert \varrho(\vert x\vert )d \mu (x) \nonumber\end{align}
for some $C''>0$ independent of $u$, which proves the claim. \par
\qed \par
We now show that the inequality obtained in Lemma \ref{1ST} implies $(S^\varrho_{1,1})$. This time we need both (\ref{L}) and (\ref{M}).

 \begin{proposition} \label{INEQ1}
Let $(X,d,\mu)$ be a quasiconvex $(DV)_{\mathrm{loc}}$ space that supports a local weak $(1,1)$-Poincar\'e inequality up to scale $R_P$. Let $N(X, \varepsilon) \ni o$ be a quasi-lattice, where $\mu(\lbrace o\rbrace)=0$ and $0 < \varepsilon \leq R_P/4$. Suppose there exists a control function $\varrho \colon [0,\infty) \rightarrow [0,\infty)$ and a constant $C > 0$ such that $$\sum_{p \in N(X, \varepsilon)} \vert v(p)\vert \mu(B(p, \varepsilon)) \leq C \sum_{p \in N(X, \varepsilon)} \sum_ {q \sim p} \vert v(p)-v(q)\vert \varrho(\vert (p,q)\vert) \mu(B(p, \varepsilon))$$ for every $v \colon N(X, \varepsilon) \rightarrow \mathbb{R}$ having finite support. Then $(X,d,\mu)$ satisfies $(S_{1,1}^\varrho)$.
\end{proposition}
\begin{proof}
Let $u \colon X \rightarrow [0,\infty)$ be a function in $N^{1,1}(X,d,\mu)$ having bounded support. Now, $$u_{B(\cdot, 4 \varepsilon)} \colon N(X, \varepsilon) \rightarrow [0, \infty)$$ is finitely supported, and since $\vert(p,q) \vert = d(\bar{o},(p,q)) \leq  2 d(o,p) + 3\varepsilon$, we have 
\begin{align}
&\sum_{p \in N(X, \varepsilon)} u_{B(p,4 \varepsilon)} \mu(B(p, \varepsilon)) \nonumber \\ &\leq C \sum_{p \in N(X, \varepsilon)} \sum_{q \sim p} \vert u_{B(p, 4 \varepsilon)} - u_{B(q, 4 \varepsilon)}\vert \varrho(\vert (p,q)\vert) \mu(B(p, \varepsilon)) \nonumber \\ \nonumber &\leq \nonumber C\sum_{p \in N(X, \varepsilon) \setminus \lbrace o \rbrace} \sum_{q \sim p} \vert u_{B(p, 4 \varepsilon)} - u_{B(q, 4 \varepsilon)}\vert \varrho(2d(o,p) + 3 \varepsilon) \mu(B(p, \varepsilon)) \\ &+ C \sum_{q \sim o} \vert u_{B(o, 4 \varepsilon)} - u_{B(q,4 \varepsilon)}\vert \varrho(3 \varepsilon) \mu(B(o, \varepsilon)). 
\nonumber 
\end{align}  
In this inequality, the first sum on the right-hand side contains every neighbour of $o$. To estimate the second sum observe that $\varrho(3 \varepsilon) \leq \varrho(2d(o,p)+3 \varepsilon)$ for every $p \in N(X,\varepsilon)$, and when $p \sim o$ we have $B(o,\varepsilon) \subseteq B(o, 4 \varepsilon) \subseteq B(p, 8 \varepsilon)$ which gives $\mu(B(o,\varepsilon)) \leq C(4 \varepsilon) C(2\varepsilon)C(\varepsilon) \mu(B(p,\varepsilon))$ using the $(DV)_{\mathrm{loc}}$ property. Put together, this gives the estimate \begin{align}
&\sum_{p \in N(X, \varepsilon)} u_{B(p,4 \varepsilon)} \mu(B(p, \varepsilon)) \nonumber \\ &\leq 2CC(4 \varepsilon)C(2\varepsilon)C(\varepsilon) \sum_{p \in N(X, \varepsilon) \setminus \lbrace o \rbrace} \sum_{q \sim p} \vert u_{B(p, 4 \varepsilon)} - u_{B(q, 4 \varepsilon)}\vert \varrho(2d(o,p) + 3 \varepsilon) \mu(B(p, \varepsilon)).\nonumber 
\end{align} 
Now, using both (\ref{L}) and (\ref{M}) this gives 
\begin{align}
&\sum_{p \in N(X, \varepsilon)} u_{B(p,4 \varepsilon)} \mu(B(p, \varepsilon)) \nonumber \\ &\leq C' \sum_{p \in N(X, \varepsilon) \setminus \lbrace o\rbrace} \sum_{q \sim p} \vert u_{B(p, 4 \varepsilon)} - u_{B(q, 4 \varepsilon)}\vert \varrho(d(o,p) ) \mu(B(p, \varepsilon))\nonumber \nonumber \\ &\leq C' \sum_{p \in N(X, \varepsilon)} \sum_{q \sim p} \vert u_{B(p, 4 \varepsilon)} - u_{B(q, 4 \varepsilon)}\vert \varrho(d(o,p) ) \mu(B(p, \varepsilon)).\nonumber 
\end{align} 
for some $C'>0$ independent of $u$. By Lemma \ref{1ST},
\begin{align}
\sum_{p \in N(X,\varepsilon)} \sum_{q \sim p} \vert u_{B(p,4\varepsilon)} - u_{B(q,4 \varepsilon)} \vert \varrho(d(o,p)) \mu(B(p,\varepsilon)) \leq C' \int_X \vert \nabla u(x)\vert \varrho(d(o,x)) d \mu (x),  \nonumber
\end{align}
so
\begin{align}
\sum_{p \in N(X, \varepsilon)} u_{B(p,4 \varepsilon)} \mu(B(p, \varepsilon)) \nonumber \leq C'' \int_X \vert \nabla u(x)\vert \varrho(d(o,x)) d \mu (x)
\end{align}
for some $C''>0$ independent of $u$.
On the other hand, by the $(DV)_\mathrm{loc}$ property 
\begin{align}
\int_X u(x) d \mu(x) \leq \sum_{p \in N(X, \varepsilon)} \int_{B(p, 4 \varepsilon)} u(x) d \mu(x) = \sum_{p \in N(X, \varepsilon)} u_{4B(p,4\varepsilon)}\mu(B(p,4 \varepsilon)) \nonumber \\ \leq C(2 \varepsilon)C(\varepsilon)\sum_{p \in N(X, \varepsilon)} u_{4B(p, \varepsilon)} \mu(B(p, \varepsilon)) \nonumber \end{align}from which the claim follows for $u \colon X \rightarrow [0,\infty)$ in $N(X,d,\mu)$ having bounded support. The claim for any $u \in N^{1,1}(X,d,\mu)$ having bounded support follows by replacing $u$ with $\vert u\vert$ and noticing that $\vert \nabla \vert u \vert\vert \leq \vert \nabla u\vert$.   \end{proof}
\subsection{From smooth to discrete}
To begin, recall the notion of Lipschitz partition of unity associated to $N(X,\varepsilon)$ and Lipschitz extensions.
\begin{definition} \cite[Section 1.12]{Heinonen2} 
A Lipschitz partition of unity associated to $N(X,\varepsilon)$ of a metric space $(X,d)$ is a locally finite family $\lbrace \varphi_p \colon p \in N(X,\varepsilon) \rbrace$ of $L$-Lipschitz functions $\varphi_p \colon X \rightarrow [0,1]$ such that $$\sum_{p \in N(X,\varepsilon)} \varphi_p(x)=1$$ for every $x \in X$ and $\varphi_p \vert (X \setminus B(p, 2 \varepsilon)) \equiv 0$.
\end{definition}
The following lemma is a modification of \cite[Section 1.12]{Heinonen2}; the proofs are essentially identical.
\begin{lemma} \label{LIPPART}
Let $(X,d)$ be a quasiconvex and uniformly coarsely proper space and $N(X,\varepsilon)$ a quasi-lattice where $0 < \varepsilon \leq 2$. Then, the family $\lbrace \varphi_p \colon p \in N(X,\varepsilon) \rbrace$ where $$\varphi_p(x) = \dfrac{\psi_p(x)}{\psi(x)},$$ $\psi_p(x)= \min \left \lbrace 1 ,\dfrac{2}{\varepsilon} \mathrm{dist} \left( x, X \setminus B(p, 3\varepsilon/ 2)\right) \right\rbrace $, and $\psi(x) = \sum_{p \in N(X,\varepsilon)} \psi_p(x)$, is a Lipschitz partition of unity associated to $N(X,\varepsilon)$.
\end{lemma} 
\begin{definition}
Let $(X,d)$ be a quasiconvex uniformly coarsely proper space and $N(X,\varepsilon)$ a quasi-lattice where $0 < \varepsilon \leq 2$. Given any function $v \colon N(X, \varepsilon) \rightarrow \mathbb{R}$, its locally Lipschitz extension $\overline{v} \colon X \rightarrow \mathbb{R}$ associated to $\lbrace \varphi_p \colon p \in N(X, \varepsilon) \rbrace$ is defined by $$\overline{v}(x) = \sum_{p \in N(X,\varepsilon)} v(p) \varphi_p(x),$$ where $\lbrace \varphi_p \colon p \in N(X, \varepsilon) \rbrace$ is the Lipschitz partition of unity associated to $N(X,\varepsilon)$. 
\end{definition}
The \emph{pointwise upper Lipschitz constant at $x \in X$} of a function $v \colon X \rightarrow \mathbb{R}$ from a metric space $(X,d)$ is $$\mathop{\mathrm{Lip}} v(x) = \limsup_{r \rightarrow 0} \sup_ {y \in B(x,r)} \dfrac{\vert v(x) - v(y)\vert}{r}.$$
Note that $\mathrm{Lip} \,\overline{v} \colon X \rightarrow [0,\infty]$ is an upper gradient of the locally Lipschitz extension $\bar{v} \colon X \rightarrow \mathbb{R}$ of $v \colon N(X,\varepsilon) \rightarrow \mathbb{R}$; see \cite[Lemma 6.2.6]{Heinonen}. We are now ready to prove the following lemma.
\begin{lemma} \label{INEQ2}
Let $(X,d,\mu)$ be a quasiconvex $(DV)_{\mathrm{loc}}$ space, $N(X, \varepsilon) \ni o$ a quasi-lattice where $0 < \varepsilon \leq 2$, $\mu(\lbrace o\rbrace)=0$, and $\varrho \colon [0,\infty) \rightarrow [0,\infty)$ a control function. Then there exists $C >  0$ such that $$\int_X \mathop{\mathrm{Lip}} \overline{v}(x) \varrho(d(o,x)) d \mu(x) \leq C \sum_{p \in N(X, \varepsilon)} \sum_{q \sim p}\vert v(p) - v(q)\vert \varrho(d(o,p)) \mu(B(p, \varepsilon))$$ for any $v \colon N(X, \varepsilon) \rightarrow \mathbb{R}$. 
\end{lemma}
\begin{proof}
Let $v \colon N(X, \varepsilon) \rightarrow \mathbb{R}$ be any function and $\overline{v} \colon X \rightarrow \mathbb{R}$ its locally Lipschitz extension as in Lemma \ref{LIPPART}. Arguing as in \cite[Lemma 3.2]{Ilkka}, there exists a constant $C>0$ such that, for any $p \in N(X, \varepsilon)$ and $x,y \in B(p, \varepsilon)$,
$$ \dfrac{\vert \overline{v}(x) - \overline{v}(y)\vert}{d(x,y)} \leq C \sum_{q \in B(p, 3 \varepsilon) \cap N(X, \varepsilon)} \vert v(q)-v(p) \vert.$$ In particular, $$\mathop{\mathrm{Lip}} \overline{v}(x) = \limsup_{r \rightarrow 0}  \sup_{y \in B(x,r)}\dfrac{\vert \overline{v}(x) - \overline{v}(y) \vert}{r} \leq C \sum_{q \in B(p, 3 \varepsilon) \cap N(X, \varepsilon)} \vert v(q)-v(p) \vert.$$ 
Thus,
\begin{align}
&\int_X \mathop{\mathrm{Lip}} \overline{v}(x) \varrho(d(o,x)) d \mu(x) \nonumber \leq \sum_{p \in N(X, \varepsilon)}\int_{B(p, \varepsilon)} \mathop{\mathrm{Lip}} \overline{v}(x) \varrho(d(o,x)) d \mu(x) \nonumber \\ &\leq C \sum_{p \in N(X, \varepsilon)} \sum_{q \in B(p, 3 \varepsilon) \cap N(X, \varepsilon)} \vert v(q)-v(p)\vert \int_{B(p, \varepsilon)} \varrho(d(o,x)) d \mu(x). \nonumber
\end{align}
The claim now follows by an application of inequality (\ref{L}). Indeed, if $x \in B(p, \varepsilon)$, then $d(o,x) \leq d(x,p) + d(p, o) \leq \varepsilon + d(o,p)$, and we have $\varrho(d(o,x)) \leq L(\varepsilon) \varrho(d(o,p))$ whenever $p \neq o$. Hence, 
\begin{align}
\int_X \mathop{\mathrm{Lip}} \overline{v}(x) \varrho(d(o,x)) d \mu(x) \leq CL(\varepsilon)\sum_{p \in N(X, \varepsilon)} \sum_{q \sim p} \vert v(q)-v(p)\vert \varrho(d(o,p)) \mu(B(p, \varepsilon))\nonumber  
\end{align} as claimed.
\end{proof}
At this point, we have the following intermediate version of \cite[Theorem 4.2]{NowakSpakula} for quasiconvex $(DV)_{\mathrm{loc}}$ spaces.

\begin{theorem} \label{SMOOTH}
If $(X,d,\mu)$ is a quasiconvex $(DV)_{\mathrm{loc}}$ space that supports a local weak $(1,1)$-Poincar\'e inequality up to scale $R_P$. Then the following are equivalent: 
\begin{enumerate}
\item[(1)] $(X,d,\mu)$ satisfies $(S_{1,1}^\varrho)$; 
\item[(2)] For any $0 < \varepsilon \leq \min \lbrace 2 , R_P/4\rbrace$ and $N(X, \varepsilon) \ni o$ such that $\mu(\lbrace o\rbrace)=0$, there exists  $C > 0$ such that $$\sum_{p \in N(X, \varepsilon)} \vert v(p)\vert \mu(B(p, \varepsilon)) \leq C \sum_{p \in N(X, \varepsilon)} \sum_{q \sim p} \vert v(p) - v(q)\vert \varrho(\vert(p,q) \vert) \mu(B(p, \varepsilon))$$ for every $v \colon N(X, \varepsilon) \rightarrow \mathbb{R}$ with finite support. 
\end{enumerate}
\end{theorem}
\begin{proof}
By Proposition \ref{INEQ1} it follows that (2) implies (1). To prove that that (1) implies (2) let $v \colon N(X, \varepsilon) \rightarrow [0,\infty)$ be finitely supported and let $\overline{v} \colon X \rightarrow [0,\infty)$ be its locally Lipschitz extension $$\overline{v}(x) = \sum_{p \in N(X,\varepsilon)} v(p) \varphi_p(x) = \sum_{p \in N(X,\varepsilon)} v(p) \dfrac{\psi_p(x)}{\psi(x)},$$
now with bounded support. Since $\overline{v}$ is locally Lipschitz, $\mathop{\mathrm{Lip}} \overline{v}$ is an upper gradient of $\overline{v}$. In particular, $\overline{v}$ and has a minimal $1$-weak upper gradient $\vert \nabla \bar{v}\vert$; see \cite[Theorem 6.3.20]{Heinonen}. Thus, by $(S_{1,1}^\varrho)$  
\begin{align}
\int_{X} \overline{v}(x) d \mu (x) \leq C \int_X \vert \nabla \overline{v}\vert \varrho(d(o,x)) d \mu (x) \leq C \int_X \mathop{\mathrm{Lip}}\overline{v}(x) \varrho(d(o,x)) d \mu(x). \nonumber
\end{align}
By Lemma \ref{INEQ2}, 
\begin{align}\int_{X} \mathop{\mathrm{Lip}} \overline{v} (x) \varrho(d(o,x)) d \mu (x) \leq C' \sum_{p \in N(X, \varepsilon)} \sum_{q \sim p}\vert v(p) - v(q)\vert \varrho(\vert (p,q)\vert) \mu(B(p, \varepsilon)). \nonumber 
\end{align}
Since $\psi$ appearing in the Lipschitz partition of unity is uniformly bounded, there exists $C''>0$ for which $\psi(x) \leq C''$ for all $x \in X$ and
\begin{align}
\int_{X} \overline{v}(x) d \mu (x) &= \int_X \sum_{p \in N(X, \varepsilon)} v(p) \varphi_p(x) d \mu (x) = \int_X \sum_{p \in N(X, \varepsilon)} v(p) \dfrac{\psi_p(x)}{\psi(x)} d \mu (x) \nonumber \\ &\geq \dfrac{1}{C''} \int_X \sum_{p \in N(X, \varepsilon)}v(p)\psi_p(x) d \mu(x) \nonumber \\  &\geq \dfrac{1}{C''} \sum_{p \in N(X, \varepsilon)} v(p) \mu(B(p, \varepsilon)),  \nonumber  
\end{align}
as $\psi_p \vert B(p, \varepsilon) \equiv 1$; and altogether for some $C''' > 0$ independent of $v$ $$\sum_{p \in N(X, \varepsilon)} v(p) \mu(B(p, \varepsilon)) \leq C''' \sum_{p \in N(X, \varepsilon)} \sum_{q \sim p}\vert v(p) - v(q)\vert \varrho(\vert (p,q)\vert) \mu(B(p, \varepsilon))$$ for every $v \colon N(X, \varepsilon) \rightarrow [0,\infty)$ with finite support. The general claim for any $v \colon N(X;\varepsilon) \rightarrow \mathbb{R}$ with finite support now follows observing that the claim holds for $\vert v \vert$ by the previous, and by the triangle inequality for $v$. 
\end{proof}

\subsection{Connecting $H^\varrho_0$ to $(S^\varrho_{1,1})$} 
Combining the previous results, we are ready to prove that the vanishing of a fundamental class in $H^\varrho_0$ of a quasiconvex $(DV)_{\mathrm{loc}}$ space that supports a local weak $(1,1)$-Poincar\'e inequality is characterised by $(S^\varrho_{1,1})$ whenever the space is uniform. We begin with the following fact. \par
\begin{lemma} \label{UNIBG}
A quasiconvex uniform space $(X,d,\mu)$ has at most exponential volume growth.
\end{lemma}
\begin{proof}
Fix a quasi-lattice $N(X,\varepsilon)$ and let $k \in \mathbb{N}\setminus \lbrace 0\rbrace$. Since $N(X,\varepsilon)$ is uniformly locally finite any open ball $B(x,2k\varepsilon) \subseteq X$ can be covered by $N(3\varepsilon)^k$ balls of radius $\varepsilon$. Since $(X,d,\mu)$ is uniform, $$\mu(B(x,2k\varepsilon)) \leq g(\varepsilon){N(3 \varepsilon)^k}$$ for every $k \in \mathbb{N}\setminus \lbrace 0\rbrace$. 
\end{proof}
\begin{theorem}
\label{ADLAD}
Let $(X,d,\mu)$ be a quasiconvex uniform space that supports a local weak $(1,1)$-Poincar\'e inequality up to scale $R_P$. Let $0 < \varepsilon \leq \min \lbrace 2, R_P/ 4\rbrace$, $N(X, \varepsilon) \ni o$, where $\mu(\lbrace o\rbrace)=0$, and $\varrho \colon [0,\infty) \rightarrow [0,\infty)$ a control function. Then, the following are equivalent:
\begin{enumerate}
\item[(1)] $(X,d,\mu)$ satisfies $(S_{1,1}^\varrho)$; 
\item[(2)] there exists $C_1 > 0$ such that for every $v \colon N(X, \varepsilon) \rightarrow \mathbb{R}$ with finite support $$\sum_{p \in N(X, \varepsilon)} \vert  v(p)\vert \mu(B(p,\varepsilon)) \leq C_1 \sum_{p \in N(X, \varepsilon)} \sum_{q \sim p} \vert v(p)-v(q)\vert \varrho(\vert (p,q)\vert) \mu(B(p, \varepsilon));$$
\item[(3)] there exists $C_2 > 0$ such that for every $v \colon N(X, \varepsilon) \rightarrow \mathbb{R}$ with finite support$$\sum_{p \in N(X, \varepsilon)} \vert  v(p)\vert \leq C_2 \sum_{p \in N(X, \varepsilon)} \sum_{q \sim p} \vert v(p)-v(q)\vert \varrho(\vert (p,q)\vert);$$ 
\item[(4)] $0 = [\Gamma] \in H^\varrho_0(\Gamma)$ for any quasi-lattice $\Gamma \subseteq X$.
\end{enumerate}
\end{theorem}
\begin{proof}
By Proposition \ref{SMOOTH}, (1) and (2) are equivalent. By uniformity $0 < f(\varepsilon) \leq \mu(B(p,\varepsilon)) \leq g(\varepsilon) < \infty$ for all $p \in N(X,\varepsilon)$, and so (2) and (3) are equivalent. Hence, it remains to prove that (3) and (4) are equivalent and we first show that (3) implies (4). First, we approximate $(X,d)$ by the space obtained from equipping $N(X,\varepsilon)$ with the edge path length $\delta \colon N(X,\varepsilon) \times N(X,\varepsilon) \rightarrow \mathbb{N} \cup \lbrace \infty \rbrace$ given by
\begin{itemize}
\item[] $\delta(x,y)=0$ if $x=y$,
\item[] $\delta(x,y) = k$ if the shortest $3 \varepsilon$-path from $x$ to $y$ is of length $k$,
\item[] $\delta(x,y)= \infty$ if there is no $3\varepsilon$-path from $x$ to $y$,
\end{itemize}
where a $3 \varepsilon$-path from $x$ to $y$ of length $k$ is any sequence of points $x=x_0, \dots, x_k=y$ in $N(X,\varepsilon)$ where $0 <d(x_i, x_{i+1}) \leq 3 \varepsilon$. Since $(X,d)$ is uniformly coarsely proper and ($Q$-)quasiconvex, $\delta$ is a metric on $N(X,\varepsilon)$ and $(N(X,\varepsilon), \delta)$ is quasi-isometric to $(X,d)$; see \cite[Proposition 3.D.16]{CorHar}, and \begin{align} \dfrac{1}{3 \varepsilon}d(q,p) \leq \delta(p,q) \leq \dfrac{Q}{\varepsilon}d(p,q) + 1 \tag{QI}\end{align} for all $p,q \in N(X,\varepsilon)$ adapting \cite[Lemma 2.5]{Kanai} for geodesic spaces to quasiconvex spaces. Thus $\varrho(d(\bar{o},(p,q))) \leq 3 \varepsilon \delta(\bar{o},(p,q))$ by (QI), and using (\ref{M}) we see that $(N(X,\varepsilon),\delta)$ satisfies 
\begin{align}\sum_{x \in N(X,\varepsilon)} \vert \eta(x)\vert \leq C_2M(3 \varepsilon)\left( \sum_{x \in \Gamma} \sum_ {\lbrace y \colon \delta(y,x)=1 \rbrace} \vert \eta(x)-\eta(y)\vert \varrho \left(\vert (x,y)\vert\right)\right) \nonumber \end{align} for every finitely supported $\eta \colon N(X,\varepsilon) \rightarrow \mathbb{R}$.
Equivalently, $0=[N(X,\varepsilon)] \in H^\varrho_0(N(X,\varepsilon))$ where $H^\varrho_0(N(X;\varepsilon))$ is defined using the metric $\delta$; see \cite[Lemma 4.1, Theorem 4.2]{NowakSpakula}. Since $\mathrm{id} \colon (N(X,\varepsilon),\delta) \rightarrow (N(X,\varepsilon),d)$ is a quasi-isometry, we conclude that $0=[(N(X,\varepsilon))] \in H^\varrho_0(N(X,\varepsilon))$, where $H_0^\varrho(N(X,\varepsilon))$ is defined using the metric $d$, and hence $0=[\Gamma] \in H^\varrho_0(\Gamma)$ for any quasi-lattice $\Gamma \subseteq X$ by Lemma \ref{vanishingql}. It remains to prove that (4) implies (3). 
By assumption, $0=[\Gamma] \in H^\varrho_0(\Gamma)$ for any quasi-lattice $\Gamma \subseteq X$; in particular for $N(X,\varepsilon) \subseteq X$. Since $\mathrm{id} \colon (N(X,\varepsilon),d) \rightarrow (N(X,\varepsilon),\delta)$ is a quasi-isometry, $0=[(N(X,\varepsilon))] \in H^\varrho_0(N(X,\varepsilon))$ defined using the metric $\delta$, equivalently, for some $D>0$  
\begin{align}\sum_{x \in N(X,\varepsilon)} \vert \eta(x)\vert \leq D\left( \sum_{x \in N(X,\varepsilon)} \sum_ {\lbrace y \colon \delta(y,x)=1 \rbrace} \vert \eta(x)-\eta(y)\vert \varrho \left(\vert (x,y)\vert\right)\right) \nonumber \end{align} for every finitely supported $\eta \colon N(X,\varepsilon) \rightarrow \mathbb{R}$. Applying (QI), (\ref{L}) and (\ref{M}), respectively, $\varrho(\delta(\bar{o},(p,q))) \leq L(1)M(Q/\varepsilon)\varrho( d(\bar{o},(p,q)))$. Using this to estimating the above inequality from above gives (3).
\end{proof}
Theorem A summarises this by stating the equivalence between (1) and (4) above. Theorem B follows from the observation that the local weak $(1,1)$-Poincar\'e inequality is not needed to prove that (1) implies (2) in Theorem \ref{SMOOTH}.  \\ \\

\bibliographystyle{abbrv}
\bibliography{kirjasto}
\end{document}